%% file: SimpleCurrent.tex
\documentclass[a4paper,11pt]{amsart}

\include{preamble}
\title{Simple current auto-equivalences of modular tensor categories}

\author{Cain Edie-Michell}
\address{Cain Edie-Michell\\
Department of Mathematics, Vanderbilt University\\
Nashville\
USA}
\email{cain.edie-michell@vanderbilt.edu}

\allowdisplaybreaks
\begin{document}

\maketitle

\begin{abstract}
In this short note we investigate the process of constructing auto-equivalences of modular tensor categories using invertible objects. We derive conditions on the invertible object for the resulting auto-equivalence to be either monoidal, braided, or pivotal. We also discuss the composition of these auto-equivalences constructed from invertible objects. To demonstrate the practicality of this construction, we construct auto-equivalences of several real-world examples of modular tensor categories.  
\end{abstract}

\section{Introduction}
Given any algebraic object, one should always attempt to study its group of symmetries. This guiding principle is particularly relevant to modular tensor categories, given the connections they provide between operator algebras, representation theory, and certain quantum field theories. A symmetry of a modular tensor category is encoded by a monoidal auto-equivalence of that category, either braided or just plain monoidal. These auto-equivalences are important ingredients in various constructions regarding modular tensor categories. In particular, braided auto-equivalences play a key role in the classification of  ``quantum subgroups'' of modular tensor categories, and are the starting point in applying the process of gauging \cite{MR3555361} which can produce new examples of modular tensor categories.

There is a process from conformal field theory known as \textit{simple current automorphisms} \cite{MR892062,MR1887583,MR1065268}, which takes an invertible object in a modular tensor category, and produces an automorphism of the fusion ring of that category. However (outside of the special case when the invertible object has order 2 \cite[Theorem 9.3]{1711.00645}) there is a gap in the literature regarding the categorical nature of these fusion ring automorphisms. In particular, it is not true in general that every fusion ring automorphism lifts to a monoidal auto-equivalence, and thus there is no guarantee that simple current automorphisms lift to monoidal auto-equivalences. In this short note we investigate the categorical nature of these simple current automorphisms.

The authors main motivation to study simple current auto-equivalences arose via the study of the auto-equivalences of the modular tensor category of level $k$ integrable representations of $\mathfrak{sl}_n$. A large class of auto-equivalences of these categories can be constructed through simple current auto-equivalences, hence it was necessary for the author to prove the results of this note. As these simple current auto-equivalences provide a very general method for constructing monoidal auto-equivalences of modular tensor categories, we choose to present this information in its own note. We anticipate it will have applications to other researchers working with modular tensor categories. 

In this note we construct monoidal structure maps for simple current automorphisms, and show that these structure maps satisfy the hexagon identity, thus showing that simple current automorphisms always lift to monoidal auto-equivalences. Further, we give necessary and sufficient conditions for these simple current auto-equivalences to be either braided, or pivotal. We summarise the main results of this note in the following Theorem.
\begin{theorem}\label{thm:main}
Let $\cC$ be a modular tensor category and let $g \in \cC$ be an invertible object of order $M$. Let $q$ be the unique eigenvalue of the braid $\sigma_{g,g}$, and fix $\zeta_M$ a primitive $M$-th root of unity such that $\zeta_M^\frac{M}{|q^2|} = q^2$. 

Suppose $\frac{M}{|q^2|} + 1$ is coprime to $M$, then there exists a monoidal auto-equivalence $\cF_{(g,\zeta_M)}$ of $\cC$ defined by
\[    \cF_{(g,\zeta_M)}(X) := g^n \otimes X,\]
where $n$ is such that $\zeta_M^n$ is equal to the unique eigenvalue of the braid $\sigma_{g,X}\circ \sigma_{X,g}$.

The auto-equivalence $\cF_{(g,\zeta_{M})}$ is braided if and only if either
\[
(M= 1,q=1,\zeta_1 = 1), \text{ or } (M= 2,q=-1,\zeta_2 = -1), \text{ or } (M= 3,q=e^\frac{\pm 2 \pi i }{3},\zeta_3 = e^\frac{\mp 2 \pi i }{3}), \text{ or }  (M= 4,q=\pm i,\zeta_4 = \mp i).
\]

The auto-equivalence $\cF_{(g,\zeta_{M})}$ is pivotal if and only if the categorical dimension of $g$ is 1.
\end{theorem}

In order to help understand the composition of these simple current auto-equivalences we also prove Proposition~\ref{prop:order}, which gives an upper bound on the order of a simple current auto-equivalence, and Proposition~\ref{prop:com}, which gives a sufficient condition for two simple current auto-equivalences to commute.

We end this note by working through several examples coming from quantum groups at integer levels. We use these examples to demonstrate the practicality of Theorem~\ref{thm:main} in constructing monoidal auto-equivalences of real-world modular tensor categories.

\section{Preliminaries}
A \textit{braided fusion category} is a fusion category, along with a collection of natural isomorphisms 
\[  \sigma_{X,Y} : X\otimes Y \to Y \otimes X\]
satisfying a certain coherence equation. We direct the reader to \cite{MR2609644} for additional details. 

We say a fusion category $\cC$ is \textit{pivotal} if there exist a collection of natural isomorphisms 
\[\psi_X: X\to X^{**}\]
that give a monoidal natural isomorphism $\Id_\cC \to (**)$.

We define the \textit{symmetric centre} of a braided fusion category $\mathcal{B}$ by
\[  \mathcal{Z}_2(\mathcal{B}) = \{ X\in \mathcal{B} : \sigma_{X,Y}\circ \sigma_{Y,X} = \id_{X\otimes Y} \text{ for all } Y\in \mathcal{B}\}.\]
A \textit{modular tensor category} is a pivotal braided fusion category whose symmetric centre is trivial.

In this note we investigate auto-equivalences of modular tensor categories. A monoidal auto-equivalence of a monoidal category $\cC$ is an abelian auto-equivalence $\cF :\cC \to \cC$, along with monoidal structure structure isomorphisms
\[   \tau_{X,Y}:  \cF(X)\otimes \cF(Y) \to \  \cF(X\otimes Y)\]
satisfying the hexagon equation
\[  
\begin{tikzcd}[row sep=3em ,column sep = .2cm]
(\cF(X)\otimes \cF(Y))\otimes \cF(Z) \arrow[rr,"\alpha_{\cF(X), \cF(Y),\cF(Z)}"] \arrow[d,"\tau_{X,Y}\otimes \id_Z"]  & \hspace{2cm}& \cF(X)\otimes (\cF(Y)\otimes \cF(Z)) \arrow[d,"\id_X \otimes \tau_{Y,Z}"]\\
\cF(X\otimes Y)\otimes \cF(Z) \arrow[d,"\tau_{X\otimes Y, Z}"] && \cF(X) \otimes (\cF(Y\otimes Z) \arrow[d,"\tau_{X,Y\otimes Z}"] \\
\cF( (X\otimes Y)\otimes Z) \arrow[rr,"\cF(\alpha_{X,Y,Z} )"] && \cF(X\otimes (Y\otimes Z)) 
\end{tikzcd} 
\]
We are also interested in braided auto-equivalences of braided tensor categories. A braided auto-equivalence of $\mathcal{B}$ is a monoidal auto-equivalence $(\cF, \tau)$ satisfying  
\[  
\begin{tikzcd}[row sep=3em ,column sep = .2cm]
\cF(X) \otimes \cF(Y) \arrow[rr,"\sigma_{\cF(X),\cF(Y)}"] \arrow[d,"\tau_{X,Y}"]  & \hspace{2cm}& \cF(Y) \otimes \cF(X) \arrow[d," \tau_{Y,X}"]\\
\cF(X\otimes Y) \arrow[rr,"\cF(\sigma_{X,Y})"] && \cF(Y\otimes X)
\end{tikzcd} 
\]
If $\cC$ is pivotal with pivotal isomorphisms $\psi_X : X \to X^{**}$ then we say a monoidal auto-equivalence $(\cF, \tau)$ is pivotal if 
\[  \delta_{X^*} \circ \cF(\psi_X) = \delta^*_X \circ \psi_{\cF(X)},\]
where $\delta_X := ((\cF(\operatorname{ev}_X)\circ \tau_{X^*,X})\otimes \id_{\cF(X)^*})\circ (\id_{\cF(X^*)}\otimes \operatorname{coev}_{\cF(X)})$.

\section{Main Results}
In this section we begin with an invertible object in a modular tensor category, and construct an abelian endo-functor. We then derive conditions on the invertible object for the endo-functor to be an auto-equivalence, a monoidal auto-equivalence, a braided auto-equivalence, or a pivotal auto-equivalence.
\vspace{1em}

Let $\cC$ be a modular tensor category, and $g \in \cC$ an invertible object of order $M$. It follows from \cite[Section 2.5]{MR1734419} that 
\[   \sigma_{g,g} = q \id_g \otimes \id_g\]
for $q$ some $2M$-th root of unity if $M$ is even, and $M$-th root of unity if $M$ is odd. In either case $q^2$ is an $M$-th root of unity, so we can define the integer
\[ A := \frac{M}{ |q^2|}. \]

Fix $\zeta_M$ a primitive $M$-th root of unity, such that 
\[    \zeta_M^{A} := q^2.\]
We remark this choice of $A$-th root of $q^2$ will have a surprisingly non-trivial effect on the resulting auto-equivalence that we are about to construct.

For each $m \in \Z{M}$ we define
\[   \cC_m := \left\{ \quad X\in \cC \quad : \quad  \raisebox{-.5\height}{ \includegraphics[scale = .5]{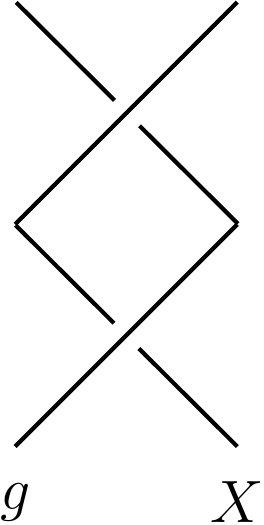}}\quad=\quad \zeta_M^m\raisebox{-.5\height}{ \includegraphics[scale = .5]{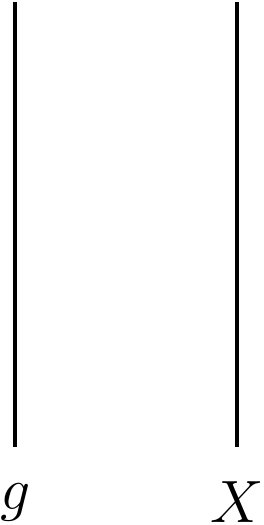}}\quad\right\}\]
 a full abelian sub-category of $\cC$. With these abelian sub-categories in hand we can write 
\[  \cC  = \bigoplus_{\Z{M}} \cC_m,\]
giving $\cC$ the structure of a $\Z{M}$-graded category. The non-degeneracy of $\cC$ ensure this grading is faithful. i.e, Let $\Z{N}$ be the supported grading group of $\cC$. Then $N$ is a divisor of $M$, and it can be checked that the object $g^N$ lives in the symmetric centre of $\cC$.

We define an abelian graded endo-functor $\cF_{(g, \zeta_M)}: \cC \to \cC$ by:
\[  \cF_{(g, \zeta_M)}(X_m) := g^m \otimes X_m \text{ for } X_m \in \cC_m,\]
and extend linearly to obtain a functor defined on all of $\cC$.

Let us investigate when this functor is an abelian auto-equivalence of $\cC$.
\begin{lemma}\label{lem:auto}
The abelian endo-functor $\cF_{(g, \zeta_M)}$ is an auto-equivalence of $\cC$ if and only if $A+1$ is coprime to $M$.
\end{lemma}
\begin{proof}
Let $X_m \in \cC_m$. The functor $\cF_{(g, \zeta_M)}$ sends $X_m$ to the object $g^m\otimes X_m$. We compute
\[ \raisebox{-.5\height}{ \includegraphics[scale = .5]{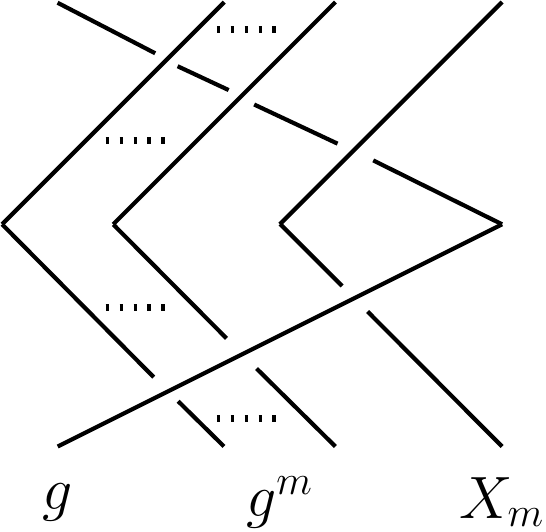}} \quad  = \quad \zeta_M^m \raisebox{-.5\height}{ \includegraphics[scale = .5]{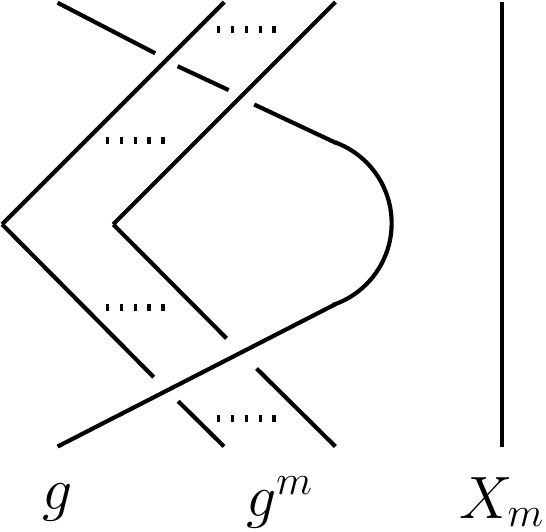}} \quad  = \quad \zeta_M^m \cdot q^{2m}\raisebox{-.5\height}{ \includegraphics[scale = .5]{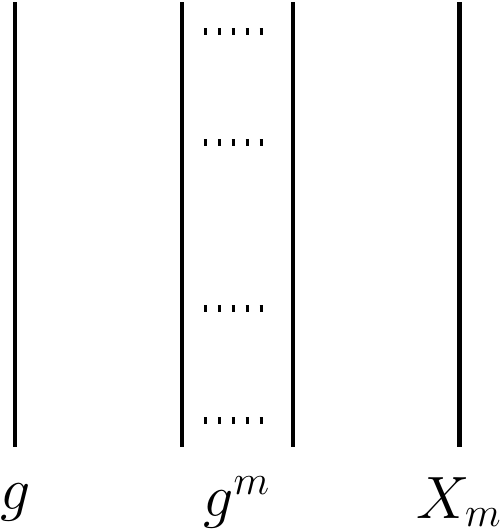}}\]
Thus, as $q^{2m} = \zeta^{Am}$, we see that $g^m\otimes X_m \in \cC_{m(A+1)}$. Hence $\cF_{(g, \zeta_M)}$ is an auto-equivalence if and only multiplication by $A+1$ induces an automorphism of the group $\Z{M}$, i.e. if and only if $A+1$ is coprime to $M$.
\end{proof}

Our goal now is to construct monoidal structure maps for the abelian auto-equivalence $\cF_{(g, \zeta_M)}$ that satisfy the hexagon equation, hence giving $\cF_{(g, \zeta_M)}$ the structure of a monoidal auto-equivalence of $\cC$. We begin by picking basis elements for the 1-dimensional Hom-spaces $\Hom( g^m \otimes g^n \to g^{m+n})$. Graphically we draw these basis elements as
\[  \raisebox{-.5\height}{ \includegraphics[scale = .5]{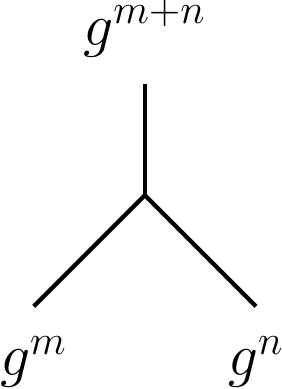}} \]
The results of \cite[Section 2.5]{MR1734419} show that we can pick these basis elements so that we have the following 6j-ology:
\[    \raisebox{-.5\height}{ \includegraphics[scale = .5]{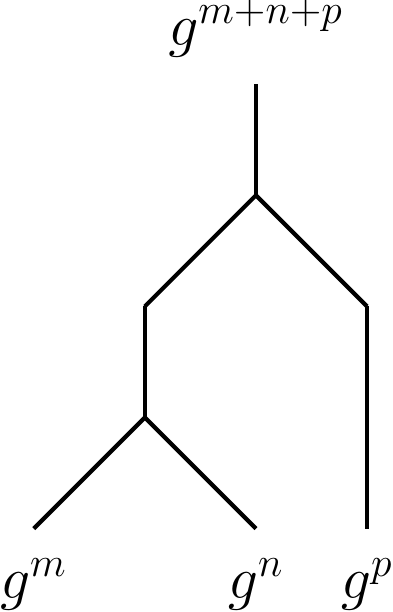}} \quad = \alpha_{m,n,p} \quad  \raisebox{-.5\height}{ \includegraphics[scale = .5]{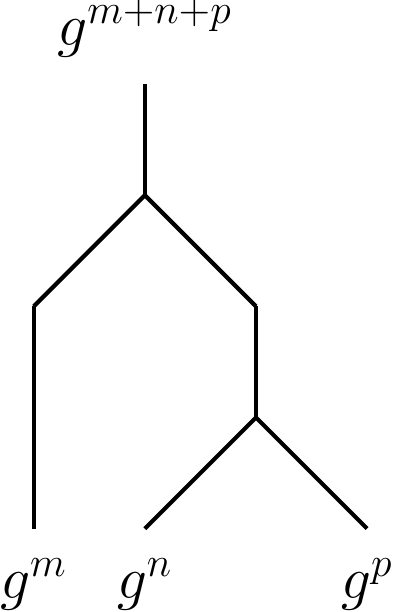}},\]
where
\[ \alpha_{m,n,p} = \begin{cases} 1 \quad &\text{ if $n+p <M$},\\ q^{Mm} \quad &\text{ if $n+p \geq M$}    \end{cases}\]
 In general we have that certain 6j symbols are non-trivial, however the following Lemma shows that we have trivial 6j symbols for the cases we care about.
 
 \begin{lemma}\label{lem:triv6}
If $A+1$ is coprime to $M$, then $q$ is an $M$-th root of unity.
\end{lemma}
\begin{proof}
If $M$ is odd then $q$ is automatically an $M$-th root of unity. If $M$ is even then $A>1$, as otherwise $A+1$ would be coprime to $M$. Thus $q$ is an $M$-th root of unity.
\end{proof}

We choose the following monoidal structure maps for $\cF_{(g, \zeta_M)}$
\[ \tau_{X_m, Y_n} :=      \raisebox{-.5\height}{ \includegraphics[scale = .5]{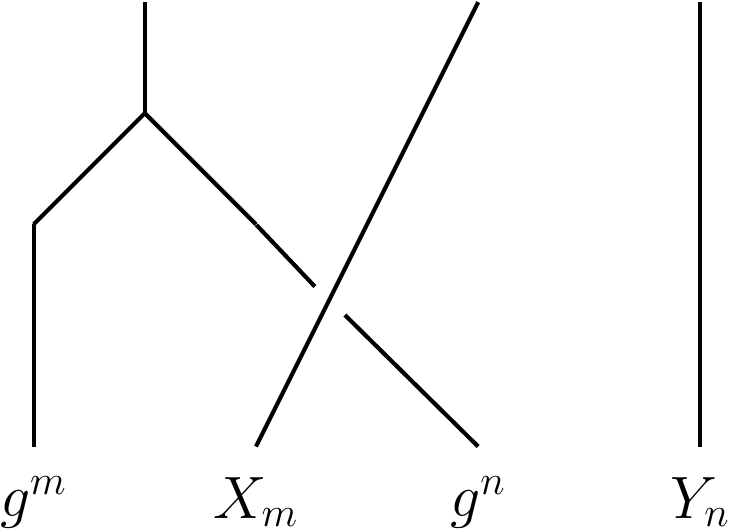}}: \cF_{(g, \zeta_M)}(X_m) \otimes \cF_{(g, \zeta_M)}(Y_n) \to \cF_{(g, \zeta_M)}(X_m \otimes Y_n).\]
We remark that one can also choose different monoidal structure maps by reversing the braiding. We state without proof that this different choice results in a naturally isomorphic monoidal auto-equivalence.

\begin{lemma}
Suppose $A+1$ is coprime to $M$, then $(\cF_{(g,\zeta_M)}, \tau)$ is a monoidal auto-equivalence.
\end{lemma}
\begin{proof}
To prove this claim we have to show that the hexagon equation is satisfied. That is we need
\[     \raisebox{-.5\height}{ \includegraphics[scale = .5]{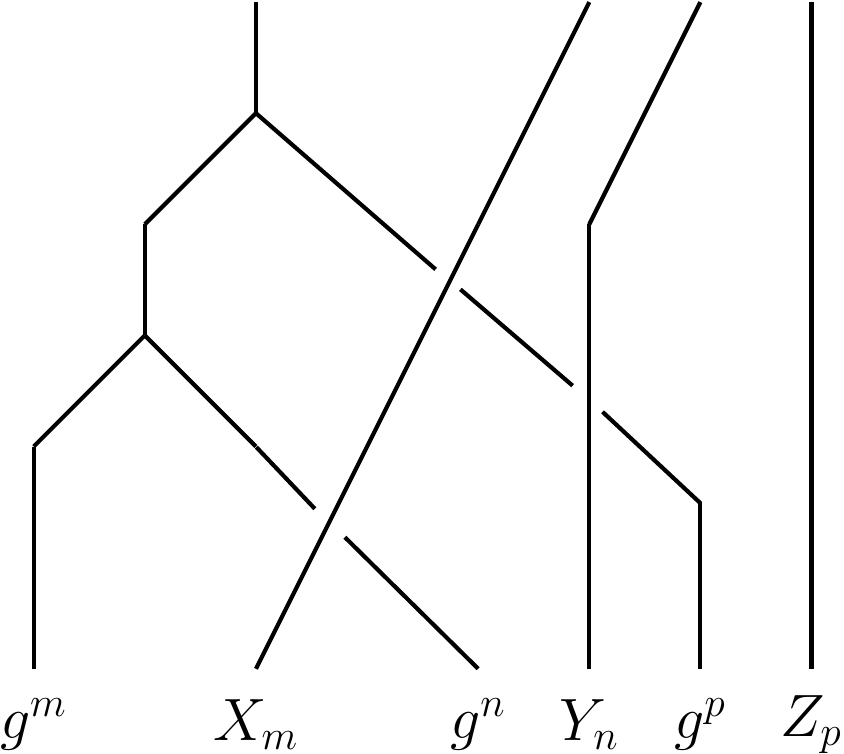}} \quad = \quad   \raisebox{-.5\height}{ \includegraphics[scale = .5]{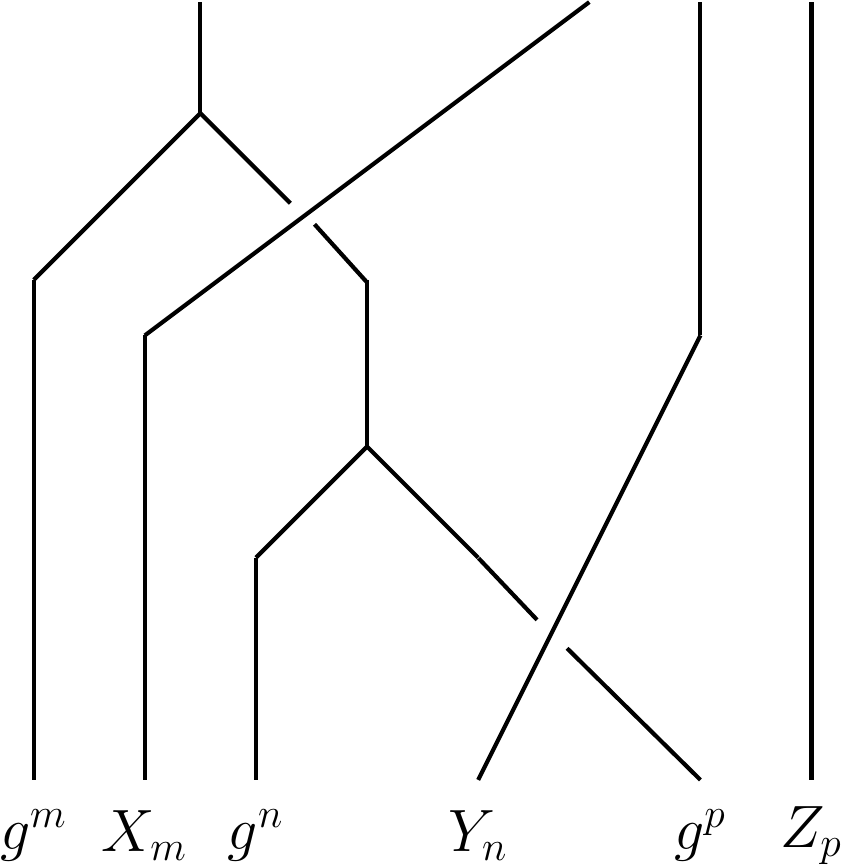}}\]
for every $X_m \in \cC_m, Y_n \in \cC_n$, and $Z_p \in \cC_p$. After an isotopy, and an application of the 6j-ology, we have that the above equations are equivalent to having $\alpha_{m,n,p} = 1$ for all $m,n,p \in \Z{M}$, which follows from Lemma~\ref{lem:triv6}.
\end{proof}

\begin{rmk}
From now on we will simply write $\cF_{(g,\zeta_M)}$ for the monoidal auto-equivalence of $\cC$, suppressing the monoidal structure maps $\tau$.
\end{rmk}

We investigate when the monoidal functor $\cF_{(g,\zeta_M)}$ is braided. For this we need to compute the R symbols for powers of the object $g$. From \cite[Section 2.5]{MR1734419}, it follows that we can arrange the trivalent vertices so that
\[     \raisebox{-.5\height}{ \includegraphics[scale = .5]{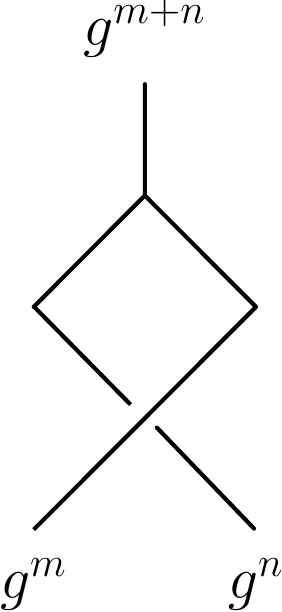}} =   q^{nm} \raisebox{-.5\height}{ \includegraphics[scale = .5]{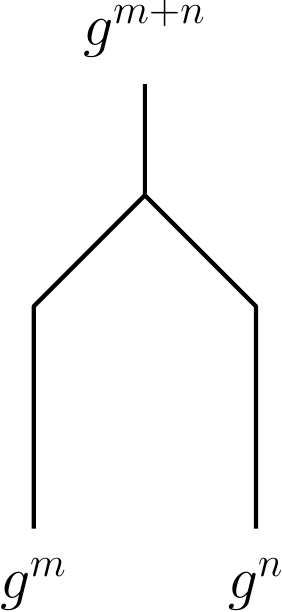}}\]
while still satisfying the 6j conditions from earlier. With these R relations, we can directly compute when $\cF_{(g,\zeta_M)}$ is braided.

\begin{lemma}
The monoidal auto-equivalence $\cF_{(g,\zeta_M)}$ is braided if and only if either
\begin{itemize}
\item $M= 1$, $q=1$ and $\zeta_1 = 1$, or

\item $M= 2$, $q=-1$ and $\zeta_2 = -1$, or

\item $M= 3$, $q= e^\frac{\pm 2 \pi i }{3}$ and $\zeta_3 = e^\frac{\mp 2 \pi i }{3}$, or

\item $M= 4$, $q=\pm i$ and $\zeta_4 = \mp i$.

\end{itemize}
\end{lemma}

\begin{proof}
For the monoidal auto-equivalence $\cF_{(g,\zeta_M)}$ to be braided we exactly need the following equality of morphisms
\[ \raisebox{-.5\height}{ \includegraphics[scale = .5]{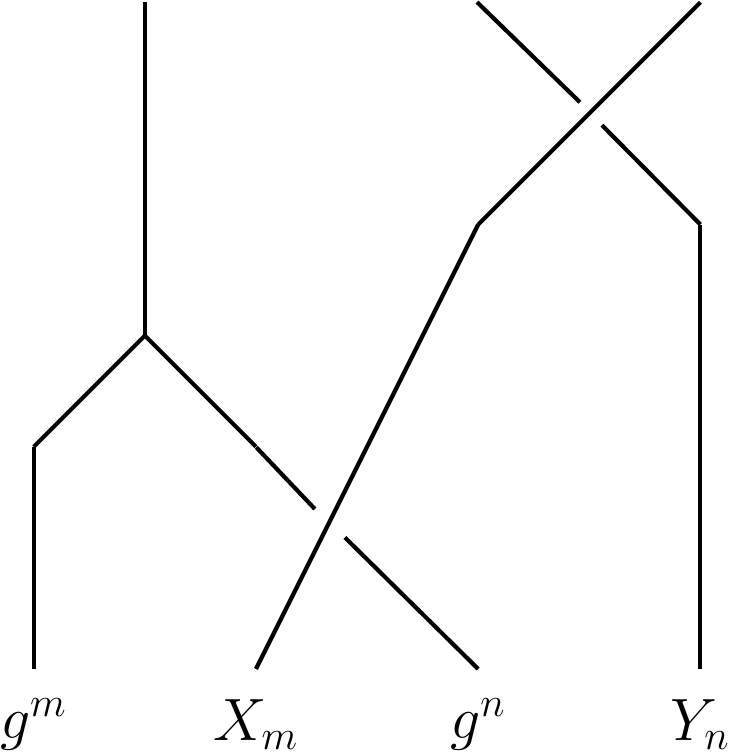}} \quad = \quad \raisebox{-.5\height}{ \includegraphics[scale = .5]{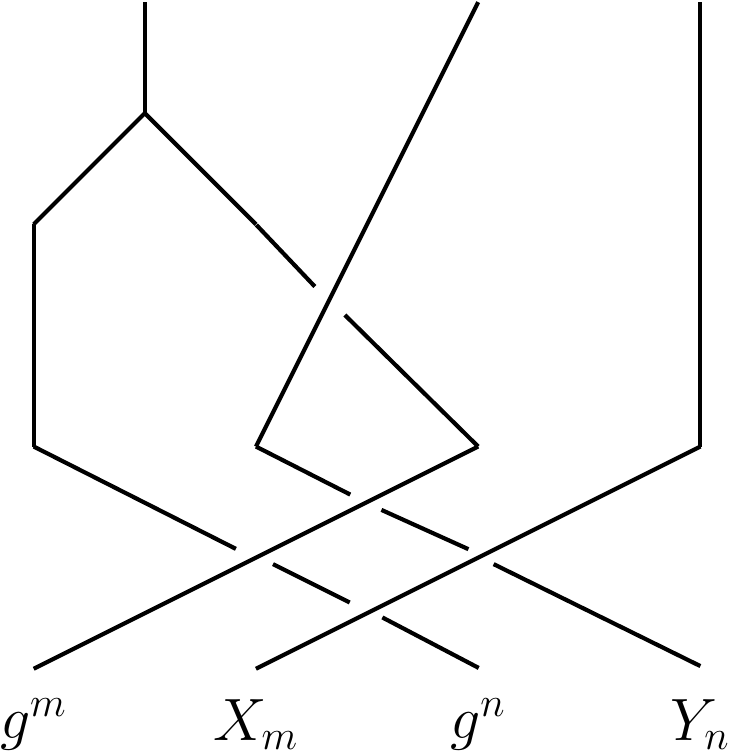}} \]
for all $X_m \in \cC_m$ and $Y_n \in \cC_n$. We can write 
\[  \raisebox{-.5\height}{ \includegraphics[scale = .5]{braided1}}  \quad = \quad \zeta_M^{mn} \raisebox{-.5\height}{ \includegraphics[scale = .5]{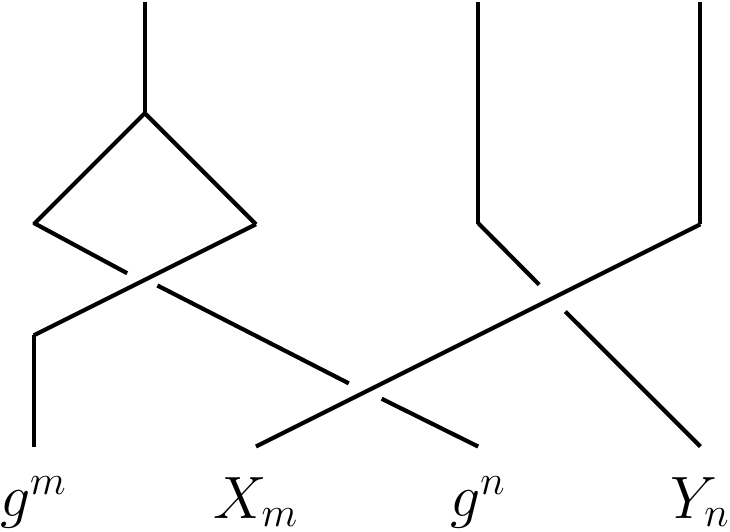}}  \quad = \quad \zeta_M^{mn}q^{mn} \raisebox{-.5\height}{ \includegraphics[scale = .5]{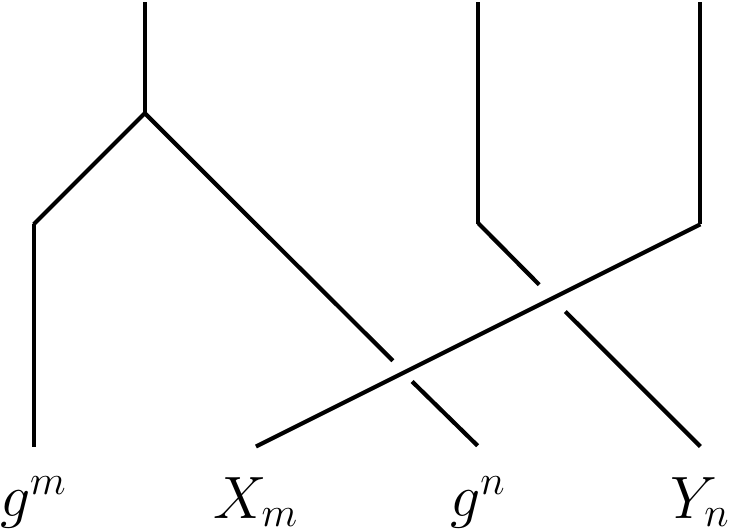}} \]
Hence the monoidal functor $\cF_{(g,\zeta_M)}$  is braided if and only if 
\[    \zeta_M^{mn}q^{mn} =1\text{ for all } m,n \in \Z{m}.\]
Setting $m = 1 = n$ shows that $\zeta_M  = q^{-1}$. Hence $q$ must be a primitive $M-th$ root of unity. Further, recall that $\zeta_M^A = q^2$. Thus we have
\[    q^{-A} = q^2,\]
and hence $2 + A$ must be a multiple of $M$. 

If $M \leq 2$ then an exhaustive search reveals the solutions 
\[ (M=1, q = 1, \zeta_1 = 1)  \quad \text{ and } \quad  (M=2, q = -1, \zeta_2 = -1). \]

If $M>2$, then $2 + A = M$, which then implies that 
\[ \frac{M}{|q^2|} = M-2.\]
Hence $M-2$ divides $M$, and so $M$ must be either $3$ or $4$. Another exhaustive search reveals the solutions 
\[   (M= 3, q = e^\frac{\pm 2 \pi i }{3}, \zeta_3 = e^\frac{\mp 2 \pi i }{3}) \quad \text{ and } \quad (M =4, q = \pm i, \zeta_4 = \mp i).\]

\end{proof}

\begin{rmk}
We note that when the order of $g$ is 3 and $q = \pm e^\frac{2 \pi i }{3}$ then $g$ generates a $\Z{3}$ modular subcategory of $\cC$. Thus $\cC$ factorises as $\Z{3} \boxtimes \cC_1$. The braided automorphism $\cF_{(g,q^{-1})}$ constructed above is exactly the braided auto-equivalence of $\Z{3} \boxtimes \cC_1$ induced by the braided auto-equivalence of the $\Z{3}$ factor.
\end{rmk}

We also investigate when the monoidal auto-equivalence $\cF_{(g,\zeta_{M})}$ is pivotal.

\begin{lemma}
The monoidal auto-equivalence $\cF_{(g,\zeta_{M})}$ is pivotal if and only if the categorical dimension of $g$ is equal to $1$.
\end{lemma}
\begin{proof}
From the definition of a pivotal functor we see that $\cF_{(g,\zeta_{M})}$ is pivotal if and only if
\[\raisebox{-.5\height}{ \includegraphics[scale = .5]{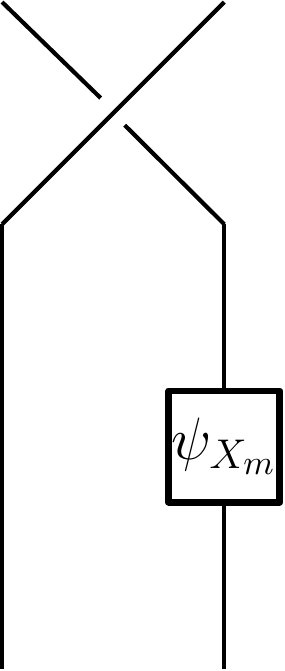}} \quad =\quad \raisebox{-.5\height}{ \includegraphics[scale = .5]{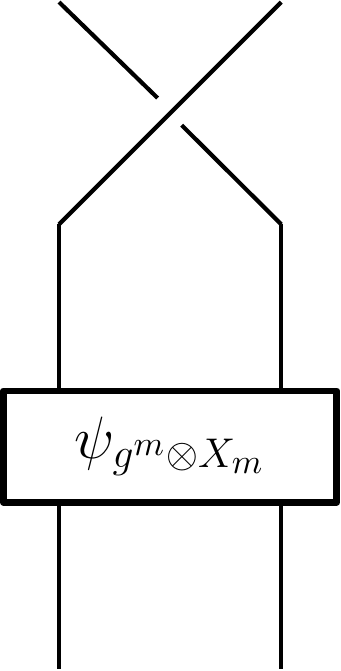}} \quad = \quad \raisebox{-.5\height}{ \includegraphics[scale = .5]{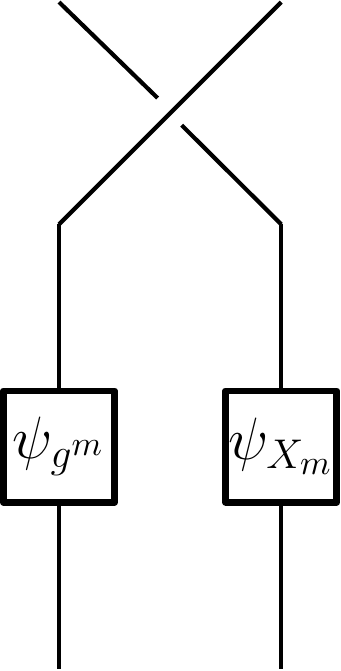}}   \]
for all $X_m \in \cC$. Hence, $\cF_{(g,\zeta_{M})}$ is pivotal if and only if $\psi_{g^m} = \id_{g^m}$ for all $m\in \Z{M}$, which is equivalent to having $\psi_g = \id_g$.

From the 6j symbols above, we can compute that the categorical dimension of $g$ is equal to $\psi_g \cdot q^M$. Hence the result follows after applying Lemma~\ref{lem:triv6}.
\end{proof}

%
%
%
%
%
%
%
%
%

 Ideally one would like to understand composition of these simple current auto-equivalences. Naively one would expect some kind of behavior such as 
 \[  \cF_{g,\zeta_{|g|}} \circ \cF_{h,\zeta_{|h|}} = \cF_{gh, \zeta_{|g|}\cdot \zeta_{|h|}}.\]
 However this naive composition does not hold in general. For an example consider an object $g$ of order $4$ with braiding eigenvalue $i$. Then we have that 
 \[  \cF_{(g, \pm i)} \circ\cF_{(g, \pm i)} \cong \Id_\cC  \quad \text{ and } \quad \cF_{(g, \pm i)} \circ\cF_{(g, \mp i)} \cong \cF_{(g^2, -1)}.\]
 Thus the auto-equivalences $\{ \Id, \cF_{(g,  i)} , \cF_{(g,-i)}, \cF_{(g^2, -1)} \}$ generically form a Klein-four group.
 
 When one considers composition between simple current auto-equivalences coming from two different invertible objects, then it appears impossible to give an answer without knowledge of a large collection of braiding eigenvalues. However, if we assume that the two invertible objects braid symmetrically, then we can prove the following fact.
 
\begin{prop}\label{prop:com}
Let $\cC$ be a modular tensor category, and $g,h$ invertible objects in $\cC$ such that $\sigma_{h,g}\sigma_{g,h} = \id_{gh}$, then
\[   \cF_{g,\zeta_{|g|}}\circ \cF_{h,\zeta_{|h|}} \cong \cF_{h,\zeta_{|h|}}\circ \cF_{g,\zeta_{|g|}}.\]
\end{prop}
\begin{proof}
Let $q_g$ be a the braiding eigenvalue of $g$, and $q_h$ be the braiding eigenvalue of $h$. As we now have two gradings on $\cC$, we write $X_{m_1, m_2}$ to indicate that $X_{m_1, m_2}$ lives in $\cC_{m_1}$ with respect to the grading induced by $g$, and lives in $\cC_{m_2}$ with respect to the braiding induced by $h$.

Let us first consider the object $ \cF_{h,\zeta_{|h|}}\circ \cF_{g,\zeta_{|g|}} (X_{m_1, m_2})$. By definition we have
\[  \cF_{h,\zeta_{|h|}}\circ \cF_{g,\zeta_{|g|}} (X_{m_1, m_2}) =   \cF_{h,\zeta_{|h|}}( g^{m_1}\otimes  X_{m_1, m_2}).\]
As the objects $g$ and $h$ braid symmetrically with each other, we then have
\[  \raisebox{-.5\height}{ \includegraphics[scale = .4]{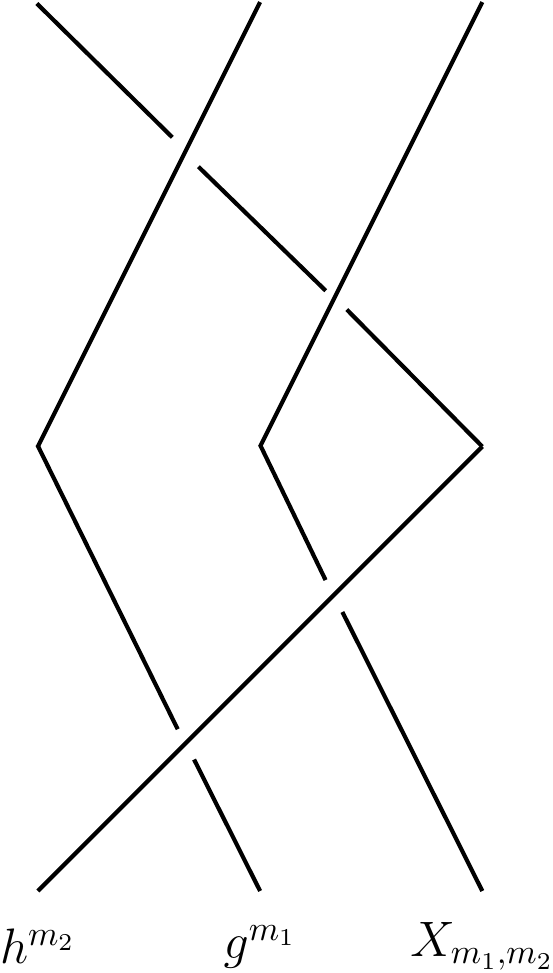}}\quad = \quad\zeta_{|h|}^{m_2} \raisebox{-.5\height}{ \includegraphics[scale = .4]{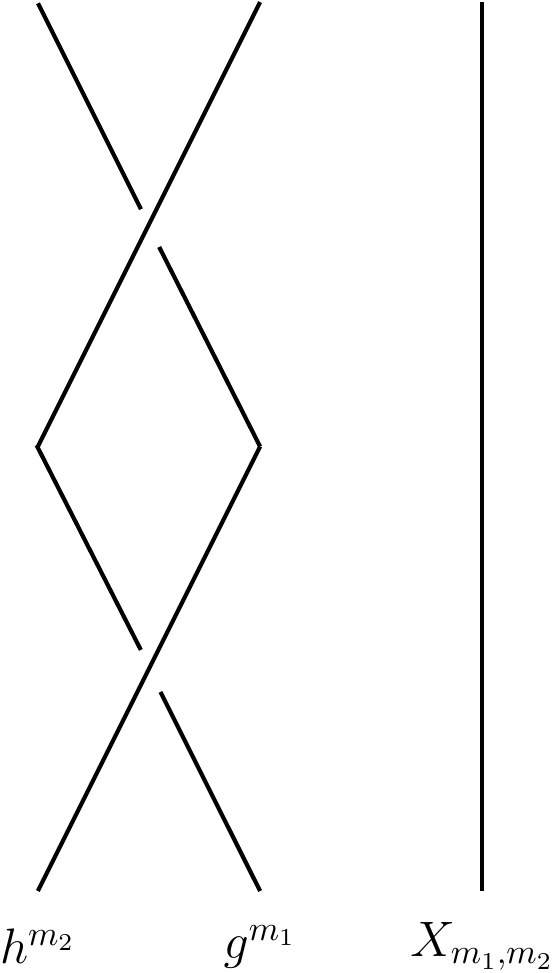}} \quad = \quad \zeta_{|h|}^{m_2}\raisebox{-.5\height}{ \includegraphics[scale = .4]{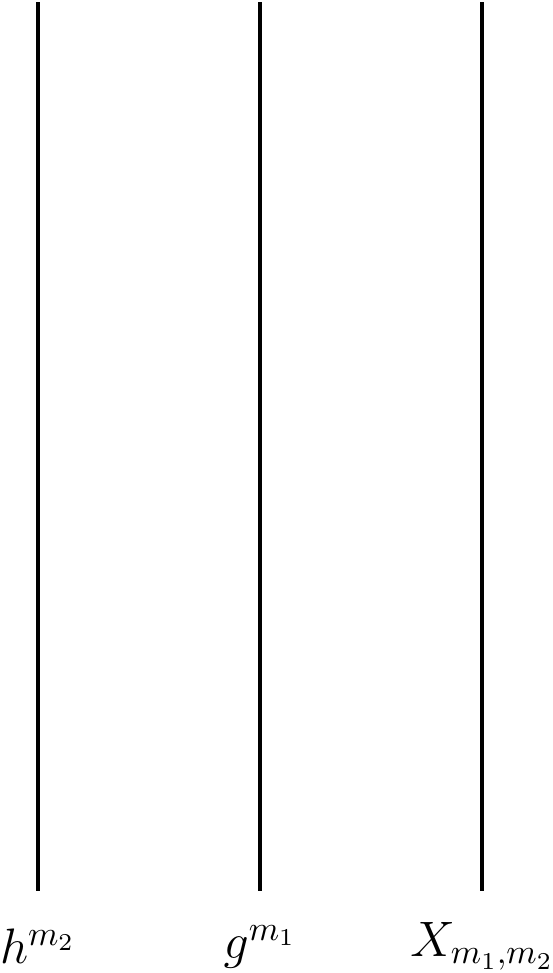}}.\]
Therefore $g^{m_1}\otimes  X_{m_1, m_2}$ lives in $\cC_{m_2}$ with respect to the grading induced by $h$. Thus
\[   \cF_{h,\zeta_{|h|}}( g^{m_1}\otimes  X_{m_1, m_2})= h^{m_2}\otimes g^{m_1}\otimes  X_{m_1, m_2}.\]

Essentially the same argument shows that 
\[\cF_{g,\zeta_{|g|}}\circ \cF_{h,\zeta_{|h|}} (X_{m_1, m_2}) = g^{m_1}\otimes h^{m_2}\otimes  X_{m_1, m_2}.\]

We construct a natural isomorphism 
\[ \eta_{X_{m_1, m_2}}:=  \raisebox{-.5\height}{ \includegraphics[scale = .5]{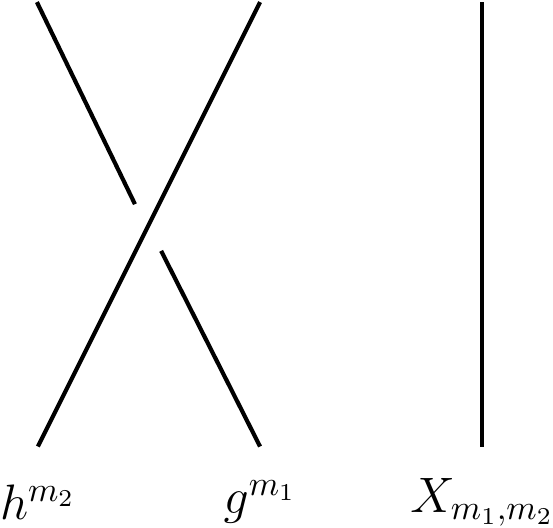}} :    \cF_{h,\zeta_{|h|}}\circ \cF_{g,\zeta_{|g|}}(X_{m_1, m_2}) \to \cF_{g,\zeta_{|g|}}\circ \cF_{h,\zeta_{|h|}}(X_{m_1, m_2}). \]

Checking that $\eta$ is monoidal amounts to showing that
\[  \raisebox{-.5\height}{ \includegraphics[scale = .5]{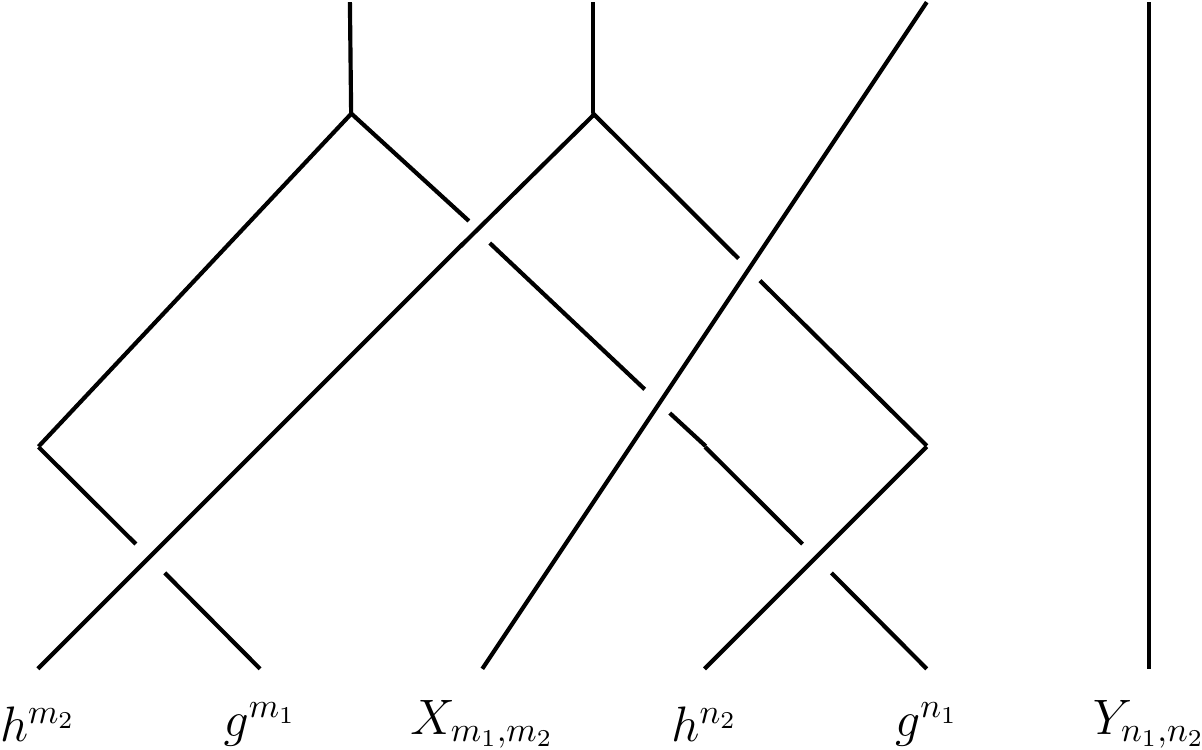}} \quad = \quad  \raisebox{-.5\height}{ \includegraphics[scale = .5]{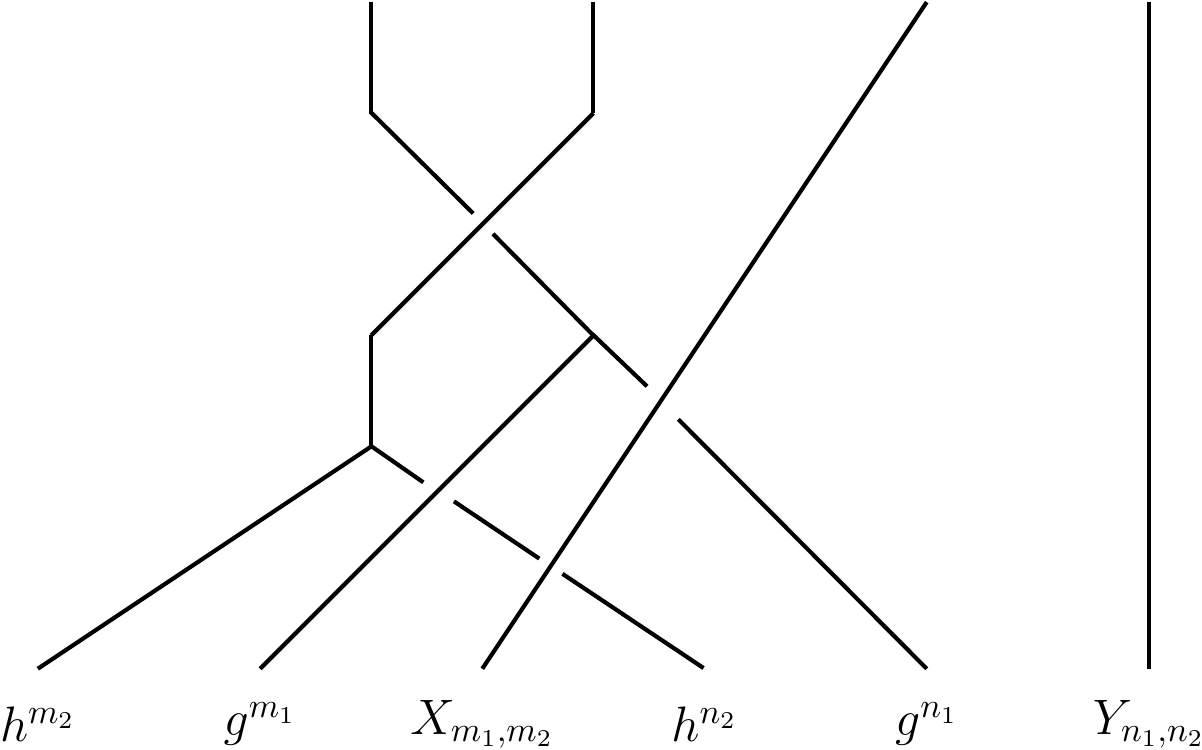}}  ,\]
which follows as $g$ and $h$ braid symmetrically.
\end{proof}
 
One would ideally like to know the order of the monoidal auto-equivalence $\cF_{g,\zeta_{|g|}}$. The following Lemma gives an upper bound for this order. The only examples we are aware of where this bound is not sharp are modular tensor categories of the form $\cC \boxtimes \mathcal{D}$, where $\cC$ is an Ising category, $\mathcal{D}$ is an arbitrary modular tensor category, and $g$ is the non-trivial invertible object in the Ising factor.

\begin{prop}\label{prop:order}
Let $K$ be an integer such that 
\[ (A+1)^K \equiv 1 \pmod {A\cdot M}\]
then 
\[ \cF_{g,\zeta_{|g|}}^K \cong \Id_\cC.\] 
\end{prop}
\begin{proof}
A similar computation to that in Lemma~\ref{lem:auto} shows that $\cF_{g,\zeta_{|g|}}$ maps $\cC_{m(A+1)^j} \to \cC_{m(A+1)^{j+1}}$. Thus, if $X_m\in \cC_m = \cC_{m(A+1)^0}$, then 
\[ \cF_{g,\zeta_{|g|}}^K(X_m) =  g^{m \sum_{j = 0}^{K-1} (A+1)^j}\otimes X_m = g^{m\frac{ (A+1)^{K} -1 }{A}}\otimes X_m.     \]

The tensor structure maps for $\cF_{g,\zeta_{|g|}}^K$, which we will denote $\tau^K$, are given by the recursive formula:
\[    \tau^1_{X_m, Y_n} =     \raisebox{-.5\height}{ \includegraphics[scale = .5]{tensor}} \quad \text{ and } \quad \tau^K_{X_m, Y_n} =  \raisebox{-.5\height}{ \includegraphics[scale = .5]{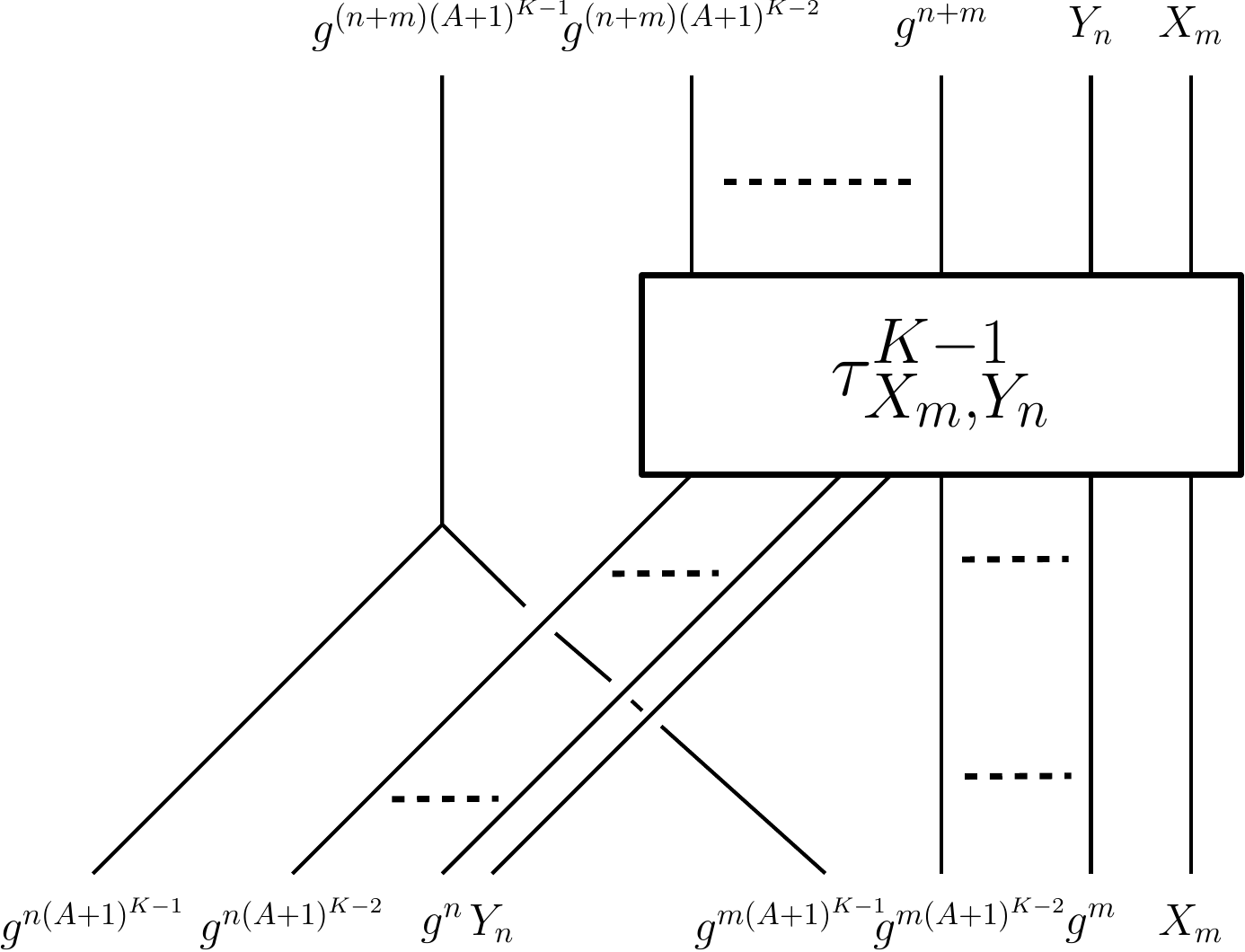}} .\]

If $K$ is an integer such that 
\[ (A+1)^K \equiv 1 \pmod {A\cdot M}\]
then $ g^{m \sum_{j = 0}^{K-1} (A+1)^j} = g^{m\frac{ (A+1)^{K} -1 }{A}} \cong \mathbf{1}$ for all $m \in \Z{M}$. In particular, the morphism
 \[ \raisebox{-.5\height}{ \includegraphics[scale = .5]{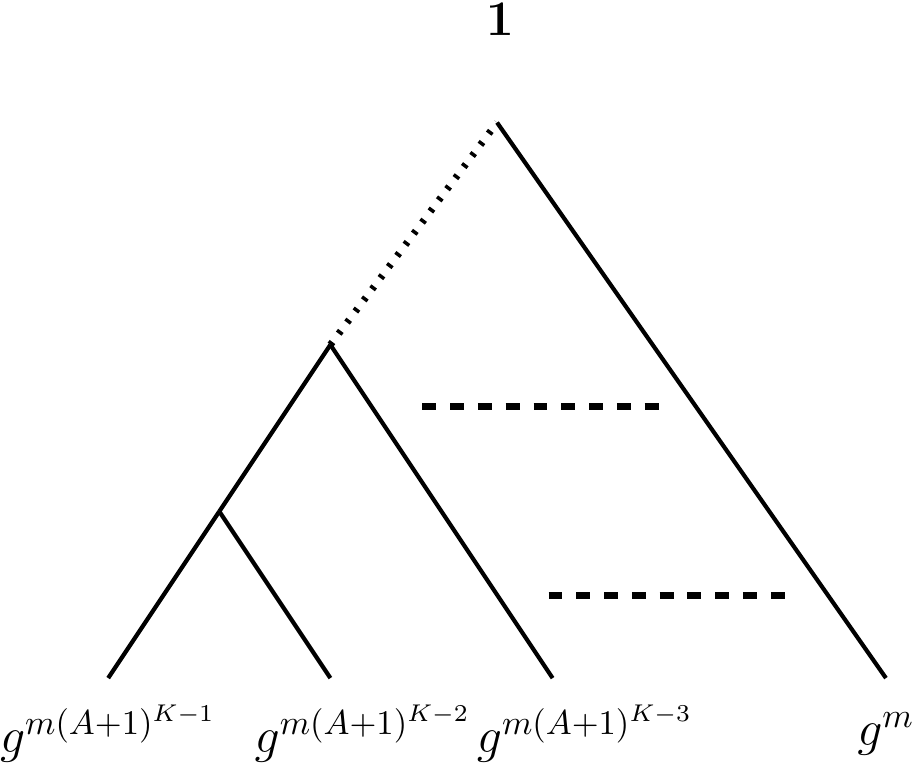}} \]
give a non-trivial map $g^{m \sum_{j = 0}^{K-1} (A+1)^j} \to \mathbf{1}$.

We claim that 
\[ \eta_{X_m} :=  \epsilon^{\frac{m^2}{2}} \raisebox{-.5\height}{ \includegraphics[scale = .5]{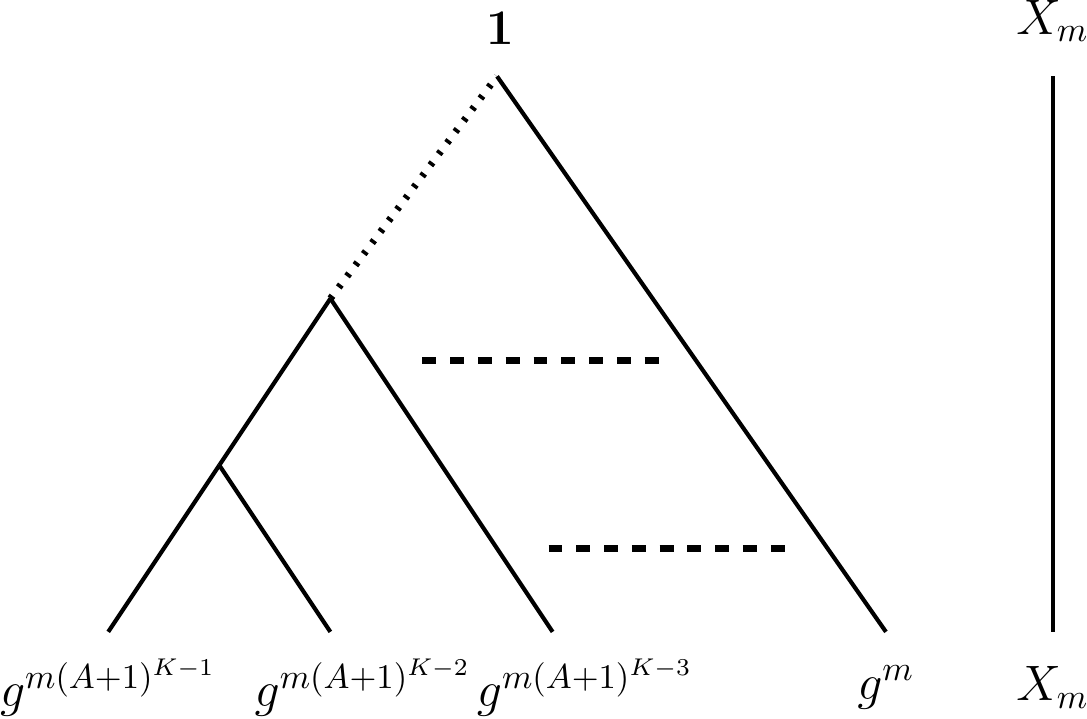}}\]
gives a monoidal natural isomorphism $\cF^K_{g, \zeta_{|g|}} \to \Id_\cC$, where 
\[\epsilon := q^{-\sum_{i =1}^{K-1} \sum_{j=i}^{2i-1} (A+1)^j  }. \]

To prove $\eta$ is a monoidal natural isomorphism we need to show that 
\[  \eta_{Y_n}\otimes \eta_{X_m} = \tau^K_{Y_n,X_m}\cdot  \eta_{Y_n \otimes X_m} \text{ for all } Y_n,X_m \in \cC.\]
Resolving the crossings in the right hand side of this equation, using the 6j-relations and the R-relations, reveals that this equation is satisfied if and only if
\[   \epsilon^{-\frac{n^2}{2}}\epsilon^{-\frac{m^2}{2}} = \epsilon^{nm} \epsilon^{-\frac{(n+m)^2}{2}} \text{ for all } m \in \Z{M}.\]
Hence we have proven that $\eta$ gives a monoidal natural isomorphism $\cF_{g,\zeta_{|g|}}^K \to \Id_\cC$.
\end{proof}

\begin{rmk} 
When $K$ is odd we have that the $\epsilon$ in the above proof is always 1.
\end{rmk}

\begin{rmk} 
By Eulers Theorem we always have the solution $K = \phi(A \cdot M)$ to 
\[ (A+1)^K \equiv 1 \pmod {A\cdot M},\]
however this solution is rarely the smallest such $K$.
\end{rmk}
%

\section{Examples}

Let us consider some examples of the practicality of Theorem~\ref{thm:main}. Our example modular tensor categories will all be categories of level $k$ integrable representations of an affine Lie algebra $\widehat{\mathfrak{g}}$, which we denote $\cC( \mathfrak{g}, k)$. For details on these categories we direct the reader to \cite{quan-primer}.

\subsection*{ $\mathfrak{sl}_4$ at level 2}

We begin with an example that illustrates the effect the choice of $\zeta_M$ has in applying Theorem~\ref{thm:main}.

We consider the modular tensor category $\cC( \mathfrak{sl}_4, 2)$. The simple objects of $\cC( \mathfrak{sl}_4, 2)$ are parametrised by triples of integers $(\lambda_1 , \lambda_2 , \lambda_3)$ such that $\lambda_1 + \lambda_2 + \lambda_3 \leq 2$. We write these simple objects as $(\lambda_1 \Lambda_1 + \lambda_2 \Lambda_2 + \lambda_3 \Lambda_3)$. In particular the category $\cC( \mathfrak{sl}_4, 2)$ contains the order 4 invertible object $2\Lambda_1$, which has braiding eigenvalue $-i$. We compute $A = 2$, thence $A+1$ is coprime to the order of $2\Lambda_1$, so we can apply Theorem~\ref{thm:main} to construct monoidal auto-equivalences of $\cC( \mathfrak{sl}_4, 2)$.

If we choose $\zeta_4 = -i$ then we get a monoidal auto-equivalence of $\cC( \mathfrak{sl}_4, 2)$ that sends 
\[ \Lambda_1 \leftrightarrow (\Lambda_1 + \Lambda_2) \quad 2\Lambda_1 \leftrightarrow 2\Lambda_3 \quad \Lambda_3 \leftrightarrow (\Lambda_2 + \Lambda_3) \]
and fixes all other objects.

If we choose $\zeta_4 = i$ then we get a braided auto-equivalence of $\cC( \mathfrak{sl}_4, 2)$ that sends
\[ \Lambda_1 \leftrightarrow \Lambda_3  \quad2\Lambda_1 \leftrightarrow 2\Lambda_3 \quad   (\Lambda_1 + \Lambda_2) \leftrightarrow (\Lambda_2 + \Lambda_3)\]
and fixes all other objects. Interestingly this braided auto-equivalence is the \textit{charge conjugation} auto-equivalence of $\cC( \mathfrak{sl}_4, 2)$. It is certainly not true in general that charge conjugation auto-equivalences can be realised as simple current auto-equivalences. An easy counter example is given by $\cC( \mathfrak{sl}_4, 4)$.

\subsection*{ $\mathfrak{sl}_6$ at level 2}

Let us now apply Theorem~\ref{thm:main} to compute the entire auto-equivalence group of a modular tensor category. Being able to compute the entire auto-equivalence group is special to this example, and in general Theorem~\ref{thm:main} will not give all auto-equivalences of a modular tensor category.

We consider the modular tensor category $\cC( \mathfrak{sl}_6,2)$. The simple objects of this category are parametrised by 5-tuples of integers $(\lambda_1 , \lambda_2 , \lambda_3 , \lambda_4, \lambda_5)$ such that $\lambda_1 + \lambda_2 + \lambda_3+ \lambda_4 + \lambda_5 \leq 2$. We write these simple objects as $(\lambda_1 \Lambda_1 + \lambda_2 \Lambda_2 + \lambda_3 \Lambda_3+ \lambda_4 \Lambda_4+ \lambda_5 \Lambda_5)$. Hence this category has 21 simple objects. In particular we have the five non-trivial invertible objects
\begin{center}
\begin{tabular}{ c | c c }
  \text{ Invertible Object} & \text{Order} & \text{Braiding Eigenvalue}  \\
  \hline \\[-1em]
$2\Lambda_1$ &6 & $e^{2\pi i \frac{5}{6}}$  \\
$2\Lambda_2$ &3 &  $e^{2\pi i \frac{1}{3}}$  \\
$2\Lambda_3$ &2 & $-1$  \\
$2\Lambda_4$ &3 & $e^{2\pi i \frac{1}{3}}$   \\
$2\Lambda_5$ &6 & $e^{2\pi i \frac{5}{6}}$  \\
\end{tabular}
\end{center}

We compute that $A+1$ is coprime to the order only for the invertible objects $2\Lambda_2$, $2\Lambda_3$, and $2\Lambda_4$. Thus we can apply Theorem~\ref{thm:main} to get the three monoidal auto-equivalences
\[  \cF_{(2\Lambda_2, e^{2\pi i \frac{2}{3}})},\qquad  \cF_{(2\Lambda_4, e^{2\pi i \frac{2}{3}})},  \quad \text{ and } \quad   \cF_{(2\Lambda_3, -1)}, \]
the last of which is a braided auto-equivalence. As there are too many simple objects to completely describe completely how each auto-equivalence behaves, we simply describe how each auto-equivalence acts on the tensor generator $\Lambda_1$. It follows from \cite{MR1237835} that this completely describes the auto-equivalence. We have
\begin{align*}
 \cF_{(2\Lambda_2, e^{2\pi i \frac{2}{3}})}( \Lambda_1 ) &= (\Lambda_2+\Lambda_3) \\
 \cF_{(2\Lambda_4, e^{2\pi i \frac{2}{3}})}( \Lambda_1) &= (\Lambda_2+\Lambda_3) \\
  \cF_{(2\Lambda_3, -1)}( \Lambda_1 ) &=  (\Lambda_3+\Lambda_4) .
\end{align*}
 A conceptionally easy, but computationally hard planar algebra argument shows that the category $\cC( \mathfrak{sl}_6,2)$ has no non-trivial monoidal auto-equivalences that fix every object. Therefore we see that 
 \[ \cF_{(2\Lambda_2, e^{2\pi i \frac{2}{3}})} \cong  \cF_{(2\Lambda_4, e^{2\pi i \frac{2}{3}})}.\]
 Further, an application of Proposition~\ref{prop:order} shows that each of these auto-equivalences has order 2. Thus to determine the subgroup structure of the auto-equivalences generated by $\cF_{(2\Lambda_2, e^{2\pi i \frac{2}{3}})}$ and $  \cF_{(2\Lambda_3, -1)}$ we simply need to determine how they compose with each other. We find that 
 \[    \cF_{(2\Lambda_2, e^{2\pi i \frac{2}{3}})}\circ \cF_{(2\Lambda_3, -1)}( \Lambda_1) = \Lambda_5 = \cF_{(2\Lambda_3,-1)}\circ    \cF_{(2\Lambda_2, e^{2\pi i \frac{2}{3}})}( \Lambda_1)  .\]
 Thus the simple current auto-equivalences  $\cF_{(2\Lambda_2, e^{2\pi i \frac{2}{3}})}$ and $  \cF_{(2\Lambda_3, -1)}$ generate a $\Z{2}\times \Z{2}$ group. We claim without proof that $\Z{2}\times \Z{2}$ is the entire auto-equivalence group of $\cC( \mathfrak{sl}_6,2)$.

\subsection*{ $\mathfrak{so}_8$ at level 2}

Our final example will consider a modular tensor category whose invertible objects form a $\Z{2}\times \Z{2}$ group. 

We consider the modular tensor category $\cC( \mathfrak{so}_8,2)$. The simple objects of this category are parametrised by 4-tuples $(\lambda_1, \lambda_2, \lambda_3, \lambda_4)$ such that $\lambda_1 + 2\lambda_2 + \lambda_3 + \lambda_4 \leq 2$. We write these simple objects as $(\lambda_1 \Lambda_1 + \lambda_2 \Lambda_2 + \lambda_3 \Lambda_3+ \lambda_4 \Lambda_4)$. This category has 11 simple objects. In particular we have the three non-trivial invertible objects $2\Lambda_1, 2\Lambda_3$, and $2\Lambda_4$. Each of these invertible objects has order 2, and braid eigenvalue $1$. An application of Theorem~\ref{thm:main} gives the three auto-equivalences
\[  \cF_{(2\Lambda_1,-1)},\qquad \cF_{(2\Lambda_3,-1)},  \quad \text{ and } \quad   \cF_{(2\Lambda_4, -1)}. \]
We directly compute how each of these auto-equivalences acts on the simple objects of $\cC( \mathfrak{so}_8,2)$:

The auto-equivalence $\cF_{(2\Lambda_1,-1)}$ sends
\[   (\Lambda_1 + \Lambda_3) \leftrightarrow \Lambda_4 \quad \text{ and } \quad \Lambda_3 \leftrightarrow (\Lambda_1+\Lambda_4),\]
and fixes all other simple objects.

The auto-equivalence $\cF_{(2\Lambda_3,-1)}$ sends
\[   (\Lambda_1 + \Lambda_3)  \leftrightarrow \Lambda_4 \quad \text{ and } \quad \Lambda_1 \leftrightarrow  (\Lambda_3 + \Lambda_4) ,\]
and fixes all other simple objects.

The auto-equivalence $\cF_{(2\Lambda_4,-1)}$ sends
\[    \Lambda_3 \leftrightarrow (\Lambda_1+\Lambda_4) \quad \text{ and } \quad  \Lambda_1 \leftrightarrow  (\Lambda_3 + \Lambda_4),\]
and fixes all other simple objects.

We state without proof that the category $\cC( \mathfrak{so}_8,2)$ has a single non-trivial auto-equivalence that fixes every object. This fact allows us to see that the auto-equivalences  
\[\{ \cF_{(2\Lambda_1,-1)},\cF_{(2\Lambda_3,-1)},   \cF_{(2\Lambda_4, -1)}\}\]
 generate a group isomorphic to either $\Z{2}\times \Z{2}$ or $\Z{4} \rtimes \Z{2}$, depending on if either $\cF_{(2\Lambda_1,-1)} \circ\cF_{(2\Lambda_3,-1)} \cong \cF_{(2\Lambda_4, -1)}$, or if $\cF_{(2\Lambda_1,-1)} \circ\cF_{(2\Lambda_3,-1)}$ is isomorphic to $\cF_{(2\Lambda_4, -1)}$ with twisted tensor structure maps. However, an application of Proposition~\ref{prop:com} shows that 
 \[ \cF_{(2\Lambda_1,-1)} \circ\cF_{(2\Lambda_3,-1)} \cong \cF_{(2\Lambda_3,-1)} \circ\cF_{(2\Lambda_1,-1)}.\]
 Thus the auto-equivalences 
 \[\{ \cF_{(2\Lambda_1,-1)},\cF_{(2\Lambda_3,-1)},   \cF_{(2\Lambda_4, -1)}\}\]
 generate a $\Z{2}\times \Z{2}$ group.

\bibliography{bibliography} 
\bibliographystyle{plain}
\end{document}

%% file: preamble.tex
\usepackage{amssymb, manfnt}
\usepackage{hyperref}
\usepackage[normalem]{ulem}
\usepackage{tikz}
\usepackage{booktabs}
\usepackage{multirow}
\usepackage{MnSymbol}
\usepackage{graphicx}
\usepackage{tikz-cd}
\usepackage{longtable}
\usepackage{amsmath,amscd}
\usepackage{tkz-graph}
\tikzstyle{vertex}=[circle, draw, inner sep=0pt, minimum size=6pt]

\usepackage[all]{xy}
\graphicspath{ {graphs/} }
\usetikzlibrary{decorations.text}
\usepackage{graphicx}
\usepackage[outdir=./pictures]{epstopdf}
\graphicspath{ {pictures/} }

\def\altdb{\vadjust{\vbox to 0pt{\vss\hbox{\kern \hsize
\quad{\dbend}}\kern\baselineskip\kern-10pt}}}

\newcommand{\arxiv}[1]{\href{http://arxiv.org/abs/#1}{\tt arXiv:\nolinkurl{#1}}}
\newcommand{\doi}[1]{\href{http://dx.doi.org/#1}{{\tt DOI:#1}}}

\newcommand{\googlebooks}[1]{(preview at \href{http://books.google.com/books?id=#1}{google books})}

\makeatletter
\let\@@pmod\pmod
\DeclareRobustCommand{\pmod}{\@ifstar\@pmods\@@pmod}
\def\@pmods#1{\mkern4mu({\operator@font mod}\mkern 6mu#1)}
\makeatother

\newcommand\id{\operatorname{id}}

\newcommand\Id{\operatorname{Id}}
\newcommand\Hom{\operatorname{Hom}}

\newcommand\Z[1]{ \mathbb{Z}_{#1}}

\newcommand\cC{ \mathcal{C}}
\newcommand\cF{ \mathcal{F}}

\theoremstyle{plain}
\newtheorem{theorem}{Theorem}[section]
\newtheorem*{theorem*}{Theorem}
\newtheorem*{prop*}{Proposition}

\newtheorem{lemma}[theorem]{Lemma}
\newtheorem{prop}[theorem]{Proposition}

\newtheorem{rmk}[theorem]{Remark}
\theoremstyle{remark}

\theoremstyle{definition}

\usepackage{color}

\numberwithin{equation}{section}

\DeclareRobustCommand*{\nicefrac}{\@UnitsNiceFrac}%
\makeatother